\documentclass[12pt]{article}
\usepackage{times}
\usepackage{amsmath,amssymb}
\usepackage{theoremref}
\usepackage[utf8]{inputenc}
\usepackage{indentfirst}
\usepackage{graphicx,tikz}
\usepackage{xcolor}
\usepackage{xskak}
\usepackage[T1]{fontenc}
\usepackage{amsthm}
\usepackage{indentfirst}

\usepackage{parskip}
\newtheorem{problem}{Question}
\newtheorem{definition}[problem]{Definition}
\newtheorem{theorem}[problem]{Theorem}

\title{Inverting Dependently Random Tournaments}
\author{Kada Williams}

\begin{document}

\maketitle

\begin{abstract}
Let $\text{inv}(T)$ be the inversion index of the tournament $T$. Answering a question of Behague, we show that $\text{inv}(T_1\to T_2)=\text{inv}(T_1)+\text{inv}(T_2)$ does not even hold asymptotically.
\end{abstract}

\pagenumbering{arabic}
\setlength{\baselineskip}{18pt}

Consider $n$ vertices, labelled by $[n]=\{1,2,\ldots,n\}$, and for every unordered pair $\{i,j\}\subset [n]$, an arrow $(i,j)$ from $i$ to $j$ or an arrow $(j,i)$ from $j$ to $i$. Such a vertex set and set of ordered pairs is known as a tournament \cite{BBB}. If the relation defined by the set of arrows is transitive, it is a transitive tournament.

A fundamental result in the study of tournaments is that a Hamiltonian directed circuit is to be found in a tournament if and only if it is not a dijoin \cite{Gho}.

\begin{definition}
    Let $T_1$ and $T_2$ be tournaments on $n_1$ and $n_2$ vertices. Their join $T_1\to T_2$ is a tournament on $n_1+n_2$ vertices, corresponding to the vertices of $T_1$ and of $T_2$, with arrows as in $T_1$ or $T_2$ and arrows from $T_1$-vertices to $T_2$-vertices.
\end{definition}

For a subset $X\subseteq [n]$, replacing the arrow $(i,j)$ with $(j,i)$ whenever $\{i,j\}\subset X$ is called an $X$-inversion. To see an example, if $X_1$ and $X_2$ are the vertex sets of the $T_1$ and $T_2$ in their join $T_1\to T_2$, inverting $X_1$, $X_2$, and $X_1\cup X_2$ yields $T_2\to T_1$.

\begin{definition}
   The least number of inversions required to reach a transitive tournament from the tournament $T$ is the inversion number $\text{inv}(T)$ of $T$.
\end{definition}

Similarly to the inversion of graphs, where non-edges are interchanged with edges \cite{BPR}, tournament inversion is intricately connected with linear algebra over $\mathbb{F}_2$ \cite{PKT}.

We biject a sequence of vertex sets $X_1,\ldots,X_k$ with an array of $n$ vectors $(\mathbf{v}_i)_{i=1}^n$, where the $t$-th component of $\mathbf{v}_i\in\mathbb{F}_2^k$ is $1$ if $i\in X_t$ and else $0$. The arrow of $\{i,j\}$ is reversed by $X_t$-inversion if and only if $(\mathbf{v}_i)_t(\mathbf{v}_j)_t=1$, so the formal dot product $\mathbf{v}_i\cdot \mathbf{v}_j=\sum_{t=1}^k(\mathbf{v}_i)_t(\mathbf{v}_j)_t$ indicates whether $(i,j)$ is reversed by these $k$ inversions.

For each $n$, is there a tournament with inversion number $n$? The answer is yes: the $n$-fold join of the $3$-cycle $C_3$. One reason is that the $\mathbb{F}_2$-vectors $\mathbf{v}_i$ are linearly independent \cite{PKT}. Such a non-elementary argument raised the question of whether the inversion number of a join is necessarily the sum of the individual inversion numbers \cite{BHS}. As it turns out, $\text{inv}(T\to C_3)=\text{inv}(T)+1$ if $\text{inv}(T)$ is even \cite{LWY}, but for every odd $k\ge 3$, there is a tournament $T_k$ with $\text{inv}(T_k)=k$ such that for all tournaments $T'$, $\text{inv}(T_k\to T')<k+\text{inv}(T')$ unless $T'$ is transitive \cite{AHH}.

In light of the trivial bounds
$$\max\{\text{inv}(T_1),\text{inv}(T_2)\}\le \text{inv}(T_1\to T_2)\le \text{inv}(T_1)+\text{inv}(T_2),$$
one could inquire into which possible values of $\text{inv}(T_1\to T_2)$ can be achieved, depending on $(\text{inv}(T_1),\text{inv}(T_2))$. By inspecting matrix rank, the lower bound cannot be attained if $\text{inv}(T_1),\text{inv}(T_2)\ge 2$ or $\text{inv}(T_1)$ is even and $\text{inv}(T_2)=1$ \cite{BJM}, \cite{BeG}. While the joins of $C_3$ are a case of equality for the upper bound, it has not been known whether $\text{inv}(T_1\to T_2)=\text{inv}(T_1)+\text{inv}(T_2)-O(1)$ \cite{Cam}. We disprove this.

\begin{theorem} \label{main}
Let $n$ be a positive integer. Then there exist tournaments $T_1$ and $T_2$ such that $\text{inv}(T_1\to T_2)\le n$, yet $\text{inv}(T_m)=n-o(n)$ for $m=1,2$.
\end{theorem}

Our proof relies on a probabilistic result from the related question of maximising $\text{inv}(T)$ across tournaments $T$ with $n$ vertices: for a uniformly random tournament $T$, $\text{inv}(T)\ge n-\sqrt{2n\log_2 n}$ with probability tending to $1$ as $n\to\infty$ \cite{APS}. We remark that the greatest $\text{inv}(T)$ achieved by a deterministic $T$, for arbitrarily large $n$, is $\lfloor \frac{n-1}{2}\rfloor$ (a transitive tournament with its unique Hamilton path reversed) \cite{Mao}.

\newpage

\begin{proof}[Proof of Theorem \ref{main}]
Suppose that $T$ and $S$ are each tournaments with $n$ vertices. It is easy to show that their inversion number must be less than $n$ \cite{BHS}, but we assume there are column vectors $\mathbf{v}_1,\ldots,\mathbf{v}_n$ and $\mathbf{u}_1,\ldots,\mathbf{u}_n$ encoding $n$ inversions that change $T$ and $S$ into the increasing linear order, respectively. By introducing a dependency between the $\mathbf{v}_i$ and $\mathbf{u}_j$, we shall make these inversions change $T\to S$ into the lexicographic order $n+1<1<n+2<2<\ldots<n$. By the probabilistic method, we can guarantee $\text{inv}(T),\text{inv}(S)=n-o(n)$.

Let the matrix with columns $\mathbf{v}_i$ and $\mathbf{u}_j$ be $A$ and $B$, respectively. The arrow from vertex $i$ of $T$ to vertex $j$ of $S$ is reversed with parity $\mathbf{v}_i\cdot \mathbf{u}_j$, and so if $\mathbf{v}_i\cdot \mathbf{u}_j=1_{i\ge j}$, then $\text{inv}(T\to S)\le n$ in the aforementioned way. Hence, we require
$$A^TB=(1_{i\ge j})_{i,j}=:L.$$

Motivated by this, let $A$ be in upper triangular form, with $A_{ii}=1$ for all $i\in[n]$ and $A_{ij}=0$ or $A_{ij}=1$ for all $i<j$ with equal chance, independently. Since the determinant of $A^T$ is $1$, it is invertible, so we can let $B=(A^T)^{-1}L$.

By applying the inversions encoded by $A$ and $B$ to the increasing linear order, we obtain the dependently random tournaments $T$ and $S$, respectively. However, $T$ and $S$ are uniformly random. Indeed, if $i<n$, switching $A_{in}$ to $1-A_{in}$ reverses the arrows precisely of the pairs $\{j,n\}$ with $A_{ij}=1$. In particular, $\{i,n\}$ is reversed and $\{j,n\}$ is not reversed if $j<i$. Hence, by double induction, all $2^{\binom n2}$ possible tournaments $T$ can arise for some value of $A$, so each of the $2^{\binom n2}$ possible values of $A$ define a unique tournament. As for $S$, the adjoints of $(A^T)^{-1}$ over the diagonal are $0$ and on the diagonal are $1$, which implies that $B$ is lower triangular with $B_{ii}=1$ for all $i\in [n]$, but $A=(LB^{-1})^T$, so $B\sim A$.

Since $T$ is uniformly random, $\text{inv}(T)<n-O(\sqrt{n\log n})$ with probability less than $\frac12$, and similarly for $S$ \cite{APS}. By the union bound, an instance of $T$ and $S$ exists for which $\text{inv}(T),\text{inv}(S)\ge n-O(\sqrt{n\log n})$, whilst $\text{inv}(T\to S)\le n$. \end{proof}

This proof was found using an LLM and communicated by the self-funded author.

\textsc{Independent author in Cambridge, United Kingdom.} \\ \\
\textit{E-mail address:} \texttt{kkw25@cantab.ac.uk}

\end{document}